\documentclass[graybox]{svmult}
\usepackage{mathptmx}       
\usepackage{helvet}         
\usepackage{courier}        
\usepackage{makeidx}         
\usepackage{graphicx}        
\usepackage{multicol}        
\usepackage[bottom]{footmisc}
\usepackage{amsmath,amssymb,amsfonts}        
\makeindex             
\newcommand{\R}{\mathbb{R}}
 
\newcommand{\N}{\mathbb{N}}

\newcommand{\Z}{{\mathbb Z}}
\newcommand{\Q}{{\mathbb Q}}

\renewcommand{\phi}{\varphi}

\newcommand{\field}[1]{\mathbb{#1}}

\newcommand{\Period}{{\rm Per}}

\newcommand{\sig}{{\rm sign}}
\newcommand{\Tr}{{\rm Tr}}
\begin{document}
\title*{A signed version of Putnam's homology theory: Lefschetz and zeta functions}
\author{Robin J Deeley}
\institute{Robin Deeley \at Department of Mathematics, University of Hawaii, 2565 McCarthy Mall, Keller 401A, Honolulu HI 96822  \email{robin.deeley@gmail.com}}
%
%
\maketitle

\abstract{A signed version of Putnam homology for Smale spaces is introduced. Its definition, basic properties and associated Lefschetz theorem are outlined. In particular, zeta functions associated to an Axiom A diffeomorphism are compared.}

\section{Introduction}
\label{sec:1}
Let $(M, f)$ be an Axiom A diffeomorphism \cite{Sma}. Then there are two natural zeta functions associated to $(M, f)$, the dynamical zeta function and the homological zeta function, see \cite[Section I.4]{Sma}. The former is defined as follows:
\[
\zeta_{{\rm dym}}(s) := \exp \left( \sum_{n\ge 1} \frac{N_n}{n} t^n \right)
\]
where $N_n$ is the cardinality of the set of points with period $n$. The definition of latter is 
\[
\zeta_{{\rm hom}}(s) := \exp \left( \sum_{n\ge 1} \frac{\tilde{N}_n}{n} t^n \right)
\]
where $\tilde{N}_n$ is obtained by counting the points of period $n$ with ``sign" (see Example \ref{ClaLefEx} or \cite[Section I.4]{Sma} for further details).

Both these functions extend meromorphically to rational functions. For the former, this is an important theorem of Manning \cite{Man}. For the latter, it is a corollary of the Lefschetz fixed point theorem. Based on Manning's result, Bowen asked whether there exists a homology theory for basic sets of an Axiom A diffeomorphism along with an associated Lefschetz theorem that implies that the dynamical zeta function is a rational function in the same way the classical Lefschetz theorem implies that the homological zeta function is a rational function. Recently, Ian Putnam constructed such a homology theory and proved the relevant Lefschetz theorem \cite{Put} (in particular see \cite[Section 6]{Put}). For more on the relationship between zeta functions, homology, and Lefschetz theorems see \cite[Section I.4]{Sma} or \cite[Section 6.1]{Put} for brief introductions or \cite{FraBook} and references therein for more details.

Putnam's homology theory is defined using the framework of Smale spaces. Smale spaces were introduced by Ruelle \cite{Rue}. The reader who is unfamiliar with them can assume that any Smale space in the present paper is either the nonwandering set or a basic set of an Axiom A diffeomorphism. The precise definition of a Smale space is given in Section \ref{SmaSpaSec}. 

It is important to note that the dynamical zeta function of an Axiom A diffeomorphism depends only on its restriction to the nonwandering set, but this is not the case for the homological zeta function. In particular, the two zeta functions defined above are not in general equal and as such the classical homology of $M$ and Putnam's homology of the nonwandering set of $M$ are (again in general) not isomorphic. 

The modest goal of the present paper is to outline the construction of a homology theory defined in the same spirit as Putnam's homology, but whose associated Lefschetz theorem is more closely related to the classical Lefschetz theorem; it counts periodic points with ``sign", see Theorem \ref{sigLefThm} for the precise statement. This goal is achieved by considering signed Smale spaces. By definition, a signed Smale space is a Smale space along with a continuous map to $\{ -1 , 1\}$, which is called a sign function. Then, by following Putnam's constructions in \cite{Put} quite closely but with this additional sign function, one obtains a new ``signed version" of Putnam's homology. In the case of an Axiom A diffeomorphism, the signed homology theory of the nonwandering set with a particular sign function is more closely related (in particular through the associated Lefschetz theorem) to the standard homology of the manifold, at least in particular situations, see Theorem \ref{eigValThm} and Example \ref{twoDimTorAut}. The notion of signed Smale space is based on work of Bowen, see in particular \cite[Theorem 2]{Bow}.

If the Smale space is connected the only possible sign functions are constant and the signed homology is essentially the same as Putnam's homology. However, typically the nonwandering set of an Axiom A diffeomorphism is not connected and this can also occur for basic sets (e.g., shifts of finite type). 

I have assumed the reader is familiar with Putnam's monograph \cite{Put} and Bowen's paper \cite{Bow}. In particular, see \cite{Bow} for more on filtrations and the no-cycle condition. Also, the reader should be warned that there are many definitions and a number proofs are omitted. Most notable among these are Theorems \ref{indPreSUpairThm} and \ref{sigLefThm}. Although, proofs of these theorems are long, the reader familiar with the proofs in \cite{Put} will likely see how they are proved. In particular, for Theorem \ref{sigLefThm}, one follows almost verbatim the construction (which is based on Manning's proof in \cite{Man}) in \cite[Section 6]{Put}. Detailed proofs of these theorems will appear elsewhere.

\section{Main results}
\label{SmaSpaSec}

\begin{definition}
A Smale space $(X,\varphi)$ consists of a compact metric space $(X, d)$ and a homeomorphism $\varphi: X \to X$ such that there exist constants $ \epsilon_{X} > 0, 0<\lambda < 1$ and a continuous partially defined map:
\[ \{(x,y) \in X \times X \mid d(x,y) \leq \epsilon_{X}\} \mapsto [x, y] \in X \]
satisfying the following axioms:
\begin{itemize}
\item[B1] $\left[ x, x \right] = x$,
\item[B2] $\left[ x, [ y, z] \right] = [ x, z]$,
\item[B3] $\left[ [ x, y], z \right] = [ x,z ]$, and
\item[B4] $\varphi[x, y] = [ \varphi(x), \varphi(y)]$;
\end{itemize}
where in these axioms, $x$, $y$, and $z$ are in $X$ and in each axiom both sides are assumed to be well-defined. In addition, $(X,\varphi)$ is required to satisfy
\begin{itemize}
\item[C1] For $x,y \in X$ such that $[x,y]=y$, we have $d(\varphi(x),\varphi(y)) \leq \lambda d(x,y)$ and
\item[C2] For $x,y \in X$ such that $[x,y]=x$, we have $d(\varphi^{-1}(x),\varphi^{-1}(y)) \leq \lambda d(x,y)$.
\end{itemize}
\end{definition}
The map $[ \, \cdot \, , \, \cdot \, ]$ in the definition of a Smale space is called the bracket map; it is unique (provided it exists). 
\begin{example}
If $(M, f)$ is an Axiom A diffeomorphism, then the restriction of $f$ to the nonwandering set is a Smale space and likewise the restriction of $f$ to a basic set is also a Smale space. The bracket map in the definition of a Smale space is, in these cases, given by the canonical coordinates.
\end{example}
An important class of Smale spaces are the shifts of finite type. They can be defined as follows. Let $G=(G^0, G^1, i, t)$ be a directed graph; that is, $G^0$ and $G^1$ are finite sets called the set of vertices and the set of edges and each edge $e \in G^1$ is given by a directed edge from $i(e) \in G^0$ to $t(e) \in G^0$, see \cite[Definition 2.2.1]{Put} for further details. 

From $G$ a dynamical system is constructed by taking 
\[
\Sigma_G := \{ (g_j)_{j\in \Z} \mid g_j \in G^1 \hbox{ and }t(g_j)=i(g_{j+1}) \hbox{ for each }j\in \Z\}
\]
with the homeomorphism, $\sigma: \Sigma_G \rightarrow \Sigma_G$ given by left sided shift. Then, see for example \cite{Put}, $(\Sigma_G, \sigma)$ is a Smale space and one can define a shift of finite type to be any dynamical system that is conjugate to $(\Sigma_G, \sigma)$ for some graph $G$. Often we will drop the $G$ from the notation and denote a shift of finite type by $(\Sigma, \sigma)$. 

From $G$ and $k\ge 2$, one can obtain a higher block presentation by constructing another graph $G^k$ whose edges are paths in $G$ of length $k$ and whose vertices are paths in $G$ of length $k-1$; for the precise details see \cite[Definition 2.2.2]{Put}.

\subsection{Signed Smale spaces}

\begin{definition}
A {\it signed Smale space} is a Smale space $(X,\varphi)$ along with a continuous map $\Delta_X: X \rightarrow \{ -1, 1\}$. Furthermore, for $n\ge 1$, we define 
$$\Delta^{(n)}_X(x)= \prod_{i=0}^{n-1} \Delta_X(\varphi^i(x)).$$
A signed Smale space is denoted by $(X, \varphi, \Delta_X)$ and $\Delta_X$ is called the sign function; it is often denoted simply by $\Delta$.
\end{definition}
\begin{example}  \label{orientExSign}
Let $(\Omega, f|_{\Omega})$ be a basic set of an Axiom A diffeomorphism, $(M, f)$. We assume that the bundle $E^u|_{\Omega}$ can be oriented and then define $ \Delta: \Omega \rightarrow \{ -1 , 1\}$ as follows: 
$$\Delta(x)= \left\{ \begin{array}{rcl} 1 & : &  D_x(f):(E^u|_{\Omega})|_x \rightarrow (E^u|_{\Omega})|_{f(x)} \hbox{ preserves the orientation } \\ -1 & : &  D_x(f):(E^u|_{\Omega})|_x \rightarrow (E^u|_{\Omega})|_{f(x)} \hbox{ reverses the orientation.} \end{array} \right.$$
The fact that $\Omega$ is hyperbolic implies that $\Delta$ is continuous; hence $(\Omega, f|_{\Omega}, \Delta)$ is a signed Smale space.
\end{example}
A special case of Example \ref{orientExSign} occurs in both the statement and proof of \cite[Theorem 2]{Bow}. Another class of examples are hyperbolic toral automorphisms: 
\begin{example}
Let $M=\R^2/\Z^2$ and $f= A$ where $A\in M_2(\Z)$ with $\det(A)=\pm 1$ and eigenvalues $\lambda_1$ and $\lambda_2$ such that $0< |\lambda_2|< 1 < |\lambda_1|$. This diffeomorphism is globally hyperbolic and the nonwandering set is the entire manifold; that is, $\Omega=M$. 

The bundle $E^u$ is isomorphic to the trivial rank one bundle. Its fiber, for example at the origin, is the eigenspace associated to $\lambda_1$. One can then show that for any $x\in M$, $\Delta(x) = {\rm sign}(\lambda_1)$.
\end{example}
\begin{definition}
Let $(\Sigma, \sigma)$ be a shift of finite type and $\Delta_{\Sigma}: \Sigma \rightarrow \{ -1, 1\}$ be a continuous function. Then $(\Sigma, \sigma, \Delta_{\Sigma})$ is called a signed shift of finite type. 
\end{definition}

\subsection{The Signed Dimension Group}

\begin{proposition} \label{graphPresentSigned}
Let  $(\Sigma, \sigma, \Delta_{\Sigma})$ be a signed shift of finite type. Then, there exists a graph $G$ such that 
\begin{enumerate}
\item there is conjugacy $h: (\Sigma_G, \sigma) \rightarrow (\Sigma, \sigma)$;
\item for any $(g_j)_{j\in \field{Z}} \in \Sigma_{G}$, $(\Delta_{\Sigma} \circ h)((g_j)_{j\in \field{Z}})$ depends only on $g_0$.
\end{enumerate}
\end{proposition}
\begin{proof}
The first item is a possible definition of a shift of finite type. Using the fact that $\Delta_{\Sigma}$ is continuous, one can obtain the second item by taking a higher block presentation.
\end{proof}
\begin{definition} \label{deltaDefGraph}
Let $(\Sigma, \sigma, \Delta_{\Sigma})$ be a signed shift of finite type and $G$ is a graph which satisfies the conclusions of the previous theorem. Then, $G$ (and the conjugacy $h: (\Sigma_G, \sigma) \rightarrow (\Sigma, \sigma)$) is called a {\it signed presentation} of $(\Sigma, \sigma, \Delta_{\Sigma})$. We denote $\Delta_{\Sigma} \circ h$ by $\Delta_{\Sigma_G}$.

By assumption, for $(g_j)_{j\in \field{Z}} \in \Sigma_G$, $\Delta_G ((g_j)_{j \in \field{Z}})$ depends only on $g_0$. As such, the function $\Delta_G: G^1 \rightarrow \{-1 , 1\}$ defined via $\Delta_G(g):=\Delta_{\Sigma_G}( (g_j)_{j\in \field{Z}})$ (where $g_0=g$) is well-defined. Moreover, for a path $g_0 \cdots g_m$ in $G$ and $n\le m$, we define 
$$\Delta_{G, n}(g_0\cdots g_m)= \prod_{j=m-n}^{m} \Delta_G(g_j).$$
Finally, given $\Delta_G$ as above, we define $\Delta_{G^k}: G^k \rightarrow \{ -1, 1\}$ via
\[
\Delta_{G^k} (g_0g_1 \cdots g_{k-1}) = \Delta_G(g_{k-1})
\] 
where $g_0g_1\cdots g_{k-1}$ is an element in $G^k$ (i.e., a path of length $k$ in $G$). This choice is based on \cite[Theorem 3.2.3 Part 1]{Put}. 
We use $(\Sigma_G, \sigma, \Delta_G)$ to denote a signed shift of finite type with a fixed signed presentation. It is important to note that $\Delta_{\Sigma_G}$ and $\Delta_G$ are related, but not the same; their domains are different.
\end{definition}
\begin{definition} \label{gammaRegSignDimGro}
Suppose $(\Sigma, \sigma, \Delta_G)$ is a signed shift of finite type with a fixed signed presentation. Define $\gamma^{s}_{G, \Delta_{G}} : \field{Z}G^0 \rightarrow \field{Z}G^0$ as follows: for each $v\in G^0$, we let
$$ v \mapsto \sum_{e \in G^1, t(e)=v} i(e) \cdot \Delta_G(e). $$
Furthermore, define $D^s_{\Delta_G}(G)$ to be the inductive limit group: $\lim_{\rightarrow} (\field{Z}G^0, \gamma^s_{G, \Delta_G})$.
\end{definition}
\begin{example}
Let $G$ be the graph with one vertex and two edges labelled by $0$ and $1$. Then the associated shift of finite type is the full two shift, $(\Sigma_G, \sigma)$. Furthermore, let $\Delta_{G} : G \rightarrow \{ -1, 1\}$ be the continuous map 
\[
\Delta_G(g) = \left\{  \begin{array}{cc} 1 & g = 1 \\ -1 & g = 0. \end{array} \right. 
\]
Then, in this case, $D^s_{\Delta_G}(\Sigma_G) \cong \{ 0 \}$.
\end{example}
\begin{theorem} (reformulation of \cite[Theorem 2]{Bow}) \label{BowResult} \\
Suppose $(M,f)$ is an Axiom A diffeomorphism satisfying the no-cycle condition, ${\rm dim}(\Omega_s)=0$, and $q:={\rm rank}(E^u|_{\Omega_s})$. Then there exists signed shift of finite type $(\Sigma_G, \sigma, \Delta_G)$ such that
\begin{enumerate}
\item $(\Sigma_G, \sigma)$ is conjugate to $(\Omega_s, f|_{\Omega_s})$;
\item the map $\gamma^s_{G, \Delta_G}: \Z G^0 \rightarrow \Z G^0$ has the same nonzero eigenvalues as the map on homology:  $f|_{M_s}: H_q(M_s, M_{s-1}) \rightarrow H_q(M_s, M_{s-1})$;
\item the map $\gamma^s_{G, \Delta_G}: D^s_{\Delta_G}(G) \rightarrow D^s_{\Delta_G}(G)$ has the same nonzero eigenvalues as the map on homology:  $f|_{M_s}: H_q(M_s, M_{s-1}) \rightarrow H_q(M_s, M_{s-1})$;
\end{enumerate}
where $(M_s)_{s=1}^m$ is a fixed filtration associated to the basic sets, $(\Omega_s)_{s=1}^m$, of $(M,f)$; we assume it satisfies the assumptions in \cite{Bow}.
\end{theorem}
\begin{proof}
Theorem 2 of \cite{Bow} implies the existence of the signed shift of finite type satisfying items (1) and (2) in the statement. Basic properties of inductive limits of abelian groups imply that for any signed shift of finite type  $\gamma^s_{G, \Delta_G}: D^s_{\Delta_G}(G) \rightarrow D^s_{\Delta_G}(G)$ and $\gamma^s_{G, \Delta_G}: \Z G^0 \rightarrow \Z G^0$ have the same nonzero eigenvalues; item (3) follows from this observation.
\end{proof}


\subsection{Signed Homology}

\begin{definition} (see \cite[Definition 2.5.5]{Put}) \\
Suppose $(X, \varphi)$ and $(Y, \psi)$ are Smale spaces and $\pi : (X, \varphi) \rightarrow (Y, \psi)$ is a factor map. Then $\pi$ is s-bijective (resp. u-bijective) if, for each $x \in X$, $\pi|_{X^s(x)}$ (resp. $\pi|_{X^u(x)}$) is a bijection to $X^s(\pi(x))$ (resp. $X^u(\pi(x))$).
\end{definition}

\begin{definition} (compare with \cite[Definition 2.6.2]{Put}) \\
Suppose $(X, \varphi, \Delta_X)$ is a signed Smale space. Then a signed s/u-bijective pair is the following data:
\begin{enumerate}
\item signed Smale spaces $(Y, \psi, \Delta_Y)$ and $(Z, \zeta, \Delta_Z)$ such that $Y^s(y)$ and $Z^u(z)$ are totally disconnected for each $y \in Y$ and $z\in Z$;
\item s-bijective map $\pi_s : (Y, \psi) \rightarrow (X, \varphi)$;
\item u-bijective map $\pi_u : (Z, \zeta) \rightarrow (X, \varphi)$;
\end{enumerate}
such that $\Delta_Y= \Delta_X \circ \pi_s$ and $\Delta_Z = \Delta_X \circ \pi_u$.
\end{definition}
\begin{proposition} (compare with \cite[Theorem 2.6.3]{Put}) \\
If $(X, \varphi, \Delta_X)$ is a nonwandering signed Smale space, then it has a signed s/u-bijective.
\end{proposition}
\begin{proof}
By \cite[Theorem 2.6.3]{Put}, $(X, \varphi)$ has an s/u-bijective pair: $(Y, \psi, \pi_s, Z, \zeta, \pi_u)$. Taking 
\[\Delta_Y:= \Delta_X \circ \pi_s \hbox{ and } \Delta_Z:= \Delta_X \circ \pi_u \]
leads to a signed s/u-bijective pair.
\end{proof}
For $L \ge 0$, $M\ge 0$, consider the Smale space (obtained via an iterated fiber product construction):

\[
\Sigma_{L, M}(\pi):= \{ (y_0, \ldots, y_L, z_0, \ldots, z_M) | \pi_s(y_i)=\pi_u(z_j) \hbox{ for each }i, j \}
\]
with $\sigma$ defined to be $\psi \times \cdots \times \psi \times \zeta \times \cdots \times \zeta$. As the notation suggests $(\Sigma_{L, M}(\pi), \sigma)$ is a shift of finte type.

Moreover, again for each $L \ge 0$, $M\ge 0$, $\Delta_{\Sigma_{L,M}(\pi)}: \Sigma_{L, M}(\pi) \rightarrow \{ -1 , 1 \}$ defined via
\[
\Delta_{\Sigma_{L,M}(\pi)}(y_0, \ldots, y_L, z_0, \ldots, z_M) =  \Delta_Y(y_0)
\]
is a continuous map. We note that $\Delta_{\Sigma_{L,M}(\pi)}(y_0, \ldots, y_L, z_0, \ldots, z_M)$ is to equal $\Delta_Y(y_i)$ for any $0\le i\le L$ and is also equal to $\Delta_{Z}(z_j)$ for any $0\le j \le M$. In particular, $\Delta_{\Sigma_{L,M}(\pi)}$ is constant on orbits of the natural action of $S_{L+1}\times S_{M+1}$. For more details on the action (which is the natural one) see \cite[Section 5.1]{Put}.

\begin{definition} 
Suppose that $\pi=(Y,\psi, \pi_s, Z, \zeta, \pi_u)$ a signed s/u-bijective pair for a signed Smale space, $(X,\varphi, \Delta)$. Then a graph $G$ is a signed presentation of $\pi$ if $G$ is a presentation of $\pi$, in the sense of Definition 2.6.8 of \cite{Put}, and $G$ is also a signed presentation, in the sense of Definition \ref{deltaDefGraph}, of $(\Sigma_{0,0}, \sigma, \Delta_{0, 0})$.
\end{definition}
\begin{proposition} (compare with \cite[Theorem 2.6.9]{Put}) \\
If $(X,\varphi, \Delta)$ is a signed Smale space and $\pi=(Y,\psi, \pi_s, Z, \zeta, \pi_u)$ is a signed s/u-bijective pair for $(X,\varphi)$, then there exists a presentation of $\pi$. Moreover, if $G$ is a signed presentation of $\pi$, then, for each $L\ge 0$ and $M\ge 0$, $G_{L,M}$ is a signed presentation of $(\Sigma_{L,M}(\pi), \sigma)$.
\end{proposition}
\begin{proof}
Work of Putnam (see \cite[Theorem 2.6.9]{Put}) implies that $\pi$ has a presentation in the sense of \cite[Definition 2.6.8]{Put}. That is, there is a graph $G$ and conjugacy $e: \Sigma_{0, 0}(\pi) \rightarrow \Sigma_G$ satisfying the conditions in \cite[Definition 2.6.8]{Put}. Moreover, since $\Delta_{\Sigma_{0,0}(\pi)}$ is continuous, by possibly taking a higher block presentation of $G$ we can ensure that this presentation leads to a signed presentation of $(\Sigma_{0,0}(\pi), \sigma, \Delta_{0,0})$. The second statement in the proposition follows as in the proof of \cite[Theorem 2.6.9]{Put} and is omitted.  
\end{proof}
\begin{definition} (compare with Definition 5.2.1 of \cite{Put}) \\
Suppose $(X,\varphi, \Delta)$ is a signed Smale space, $\pi=(Y,\psi, \pi_s, Z, \zeta, \pi_u)$ a signed s/u-bijective pair for $(X,\varphi)$, and $G$ is a presentation of $\pi$. Fix $k\ge 0$, $L \ge 0$, and $M\ge 0$ and let
\begin{enumerate}
\item $\mathcal{B}(G^k_{L, M}, S_L \times 1)$ be the subgroup of $\field{Z}G^k_{L,M}$ which is generated by elements of the following forms:
\begin{enumerate}
\item $p \in G^k_{L,M}$ with the property that $p \cdot ( (\alpha, 1)=p$ for some non-trivial transposition, $\alpha \in S_{L+1}$;
\item $p^{\prime}=q\cdot (\alpha, 1) - \sig(\alpha)q$ for some $q\in G^k_{L,M}$ and $\alpha \in S_{L+1}$;
\end{enumerate}
\item $\mathcal{Q}(G^k_{L,M}, S_L \times 1)$ be the quotient of $\field{Z}G^k_{L,M}$ by $\mathcal{B}(G^k_{L,M})$; we denote the quotient map by $Q$;
\item $\mathcal{A}(G^k_{L,M}, 1 \times S_{M+1})$ be $\{ a \in \field{Z}G^k_{L,M} \ | \ a \cdot (1, \beta)= \sig(\beta)\cdot a \hbox{ for all } \beta \in S_{M+1} \}$; it is a subgroup of $\field{Z}G^k_{L,M}$.
\end{enumerate}
\end{definition}
\begin{proposition} (see the remark between Definitions 5.2.1 and 5.2.2 in \cite{Put}) \\
Suppose $\pi=(Y,\psi, \pi_s, Z, \zeta, \pi_u)$ a signed s/u-bijective pair for a signed Smale space, $(X,\varphi, \Delta)$ and $G$ is a signed presentation of $\pi$. Then, for each $k\ge 0$, $L \ge 0$, and $M\ge 0$,
\begin{eqnarray*}
\gamma^s_{G^k_{L,M}, \Delta_{G^k_{L,M}}} (\mathcal{B}(G^k_{L, M}, S_L \times 1)) & \subseteq & \mathcal{B}(G^k_{L, M}, S_L \times 1) \\
\gamma^s_{G^k_{L,M}, \Delta_{G^k_{L,M}}} (\mathcal{A}(G^k_{L, M}, S_L \times 1)) & \subseteq & \mathcal{A}(G^k_{L, M}, S_L \times 1)
\end{eqnarray*}
where $\gamma^s_{G^k_{L,M}, \Delta_{G^k_{L,M}}}$ is defined in Definition \ref{gammaRegSignDimGro}.
\end{proposition}
\begin{definition} (compare with \cite[Definition 5.2.2]{Put}) \\
Suppose $\pi=(Y,\psi, \pi_s, Z, \zeta, \pi_u)$ a signed s/u-bijective pair for a signed Smale space, $(X,\varphi, \Delta)$ and $G$ is a signed presentation of $\pi$. Using the previous proposition, we define
$$D^s_{Q,A,G^k, \Delta_{G^k}}(G^k_{L,M})= \lim_{\rightarrow} \left(Q(A(G^k_{L,M}, 1\times S_{M+1})), \gamma^s_{G^k_{L,M}, \Delta_{G^k_{L,M}}} \right)  $$ 
\end{definition}
For each $0\le i \le L$, there is a map defined at the level of graphs, $\delta^s_{i, }: G^k_{L, M} \rightarrow G^k_{L-1, M}$ obtained by removing the $i^{{\rm th}}$ entry. Likewise, for $0\le j\le M$, one has a map $\delta^s_{,j }: G^k_{L, M} \rightarrow G^k_{L, M-1}$ that is defined by removing the $L+j$-entry. As in \cite{Put}, these induce maps at the level of the abelian groups introduced in the previous definition:
\begin{proposition} (compare with \cite[Lemma 5.2.4]{Put}) \\
Suppose $\pi=(Y,\psi, \pi_s, Z, \zeta, \pi_u)$ a signed s/u-bijective pair for a signed Smale space, $(X,\varphi, \Delta)$ and $G$ is a signed presentation of $\pi$. Then, there exists $k\in \N$ such that $\delta_{i, }$ and $\delta_{, j}$ induced group homomorphisms:
\[ \delta^s_{i, }:  D^s_{Q,A,G^k, \Delta_{G^k}}(G^k_{L,M}) \rightarrow D^s_{Q,A,G^k, \Delta_{G^k}}(G^k_{L-1,M}) \]
and
\[ \delta^{s*}_{,j}: D^s_{Q,A,G^k, \Delta_{G^k}}(G^k_{L,M}) \rightarrow D^s_{Q,A,G^k, \Delta_{G^k}}(G^k_{L,M+1})\]
respectively.
\end{proposition}
\begin{definition} (compare with \cite[Definition 5.1.7]{Put} and \cite[Sections 5.2 and 5.3]{Put}) \\
Suppose $\pi=(Y,\psi, \pi_s, Z, \zeta, \pi_u)$ is a s/u-bijective pair for a signed Smale space, $(X,\varphi, \Delta)$, $G$ is a signed presentation of $\pi$, and $k$ is as in the statement of previous proposition. Then, we let
$$d^s_{Q,A,G^k, \Delta_{G^k}}(\pi)_{L,M}: D^s_{Q, A, \Delta_{G^k}}(G^k_{L,M}) \rightarrow D^s_{Q, A, \Delta_{G^k}}(G^k_{L-1, M}) \oplus D^s_{Q,A, \Delta_{G^k}}(G^k_{L, M+1})$$
be the map 
$$\sum_{i=0}^{L} (-1)^i \delta^s_{i, } + (-1)^L \sum_{j=0}^{M} (-1)^j \delta^{s*}_{,j}. $$ 
Finally, for each $N\in \mathbb{Z}$, we let $d^s_{Q,A,G^k, \Delta_{G^k}}(\pi)_N = \bigoplus_{L-M=N} d^s_{Q,A,G^k, \Delta_{G^k}}(\pi)_{L,M}$.
\end{definition}

\begin{theorem} (see \cite[Sections 5.1 and 5.2]{Put}) \\
Assuming the setup of the previous definition, 
\[
\left( \bigoplus_{L-M=N} D^s_{Q,A,\Delta_{G^k_{L,M}}}(G_{L,M}), \bigoplus_{L-M=N} d^s_{Q,A,G^k, \Delta_{G^k}}(\pi)_{L,M} \right)_{N\in \field{Z}} 
\]
is a complex.
\end{theorem}

\begin{definition} (compare with \cite[Definition 5.1.11]{Put}) \\
Suppose $\pi=(Y,\psi, \pi_s, Z, \zeta, \pi_u)$ is a s/u-bijective pair for a signed Smale space, $(X,\varphi, \Delta)$ and $G$ is a signed presentation of $\pi$. We define $H^s_*(X,\varphi, \Delta, \pi, G^k)$ to be the homology of the complex
$$\left( \bigoplus_{L-M=N} D^s_{Q,A,\Delta_{G^k_{L,M}}}(G_{L,M}), \bigoplus_{L-M=N} d^s_{Q,A,G^k, \Delta_{G^k}}(\pi)_{L,M} \right)_{N\in \field{Z}} $$
from the previous theorem. We call this the signed homology and denote it by $H^s_*(X, \varphi, \Delta, \pi, G^k)$; it is a $\Z$-graded abelian group.
\end{definition}

\begin{theorem} (compare with \cite[Theorem 5.1.12]{Put}) \\
The signed homology groups have finite rank and vanish for all but finitely many $N\in \Z$. 
\end{theorem}
\begin{proof}
Basic properties of inductive limits imply that the signed dimension groups have finite rank. Hence the homology is finite rank (see for example page 131 of \cite{Put} for further details). That the homology vanishes for all but finitely many $N$ also follows as in \cite{Put} pages 131-132.
\end{proof}
\begin{theorem} (compare with \cite[Theorem 5.5.1]{Put}) \label{indPreSUpairThm} \\
The signed homology is independent of the choice of signed presentation, and the choice of s/u-bijective pair.
\end{theorem}
\begin{definition}
Suppose $(X, \varphi, \Delta)$ is a signed Smale space. Based on the previous theorem, for any choice of signed s/u-bijective pair, $\pi$, signed presentation $G$, and $k$ large enough, we can define $H^s_N(X, \varphi, \Delta):=H^s_N(X,\varphi, \Delta, \pi, G^k)$.
\end{definition}
\begin{proposition} (compare with a special case of \cite[Theorem 5.4.1]{Put}) \\
Suppose $(X, \varphi, \Delta)$ is a signed Smale space. The homeomorphism $\varphi$ and its inverse induces graded group homomorphism at the level of the signed homology groups. We denote the induced maps by $\varphi^s$ and $(\varphi^{-1})^s$ respectively.
\end{proposition}
\begin{remark}
General functorial properties Putnam's homology theory are nontrivial, see \cite{DKW, DKWcor, Put}. The functorial properties of the signed version are further complicated by the requirement that the map at the level of Smale space must respect the signed structure. The full details of these properties are not discussed here as they are not needed for the signed Lefschetz theorem.
\end{remark}
\begin{example}
 Suppose $(X, \varphi)$ is a Smale space and we take $\Delta_X$ to be the constant function one. Then, it follows from the defintions involved that $H^s(X, \varphi, \Delta)$ is Putnam's stable homology theory.
\end{example}
\begin{example}
Suppose $(\Sigma_G,  \sigma, \Delta_G)$ is a signed shift of finite type. The signed homology, $H^s_N(\Sigma_G, \sigma, \Delta_G)$ is the signed dimension group when $N=0$ and is the trivial group when $N\neq 0$. 
\end{example}

\subsection{Lefschetz and zeta functions}
\begin{definition}
Suppose $(X, \varphi)$ is a Smale space. Then, for each $n \in \N$,
\[ 
\Period(X, \varphi, n) :=  \{ x \in X \mid \varphi^n(x) =x \}. 
\]
\end{definition}
\begin{definition} \label{signedDynZeta}
Suppose $(X, \varphi, \Delta)$ is a signed Smale space. Then, the signed dynamical zeta function is 
\[
\zeta_{(X, \Delta)}(z) = \exp \left( \sum_{n=1}^{\infty} \frac{N_n(X, \varphi, \Delta)}{n} z^n \right)
\]
where $N_n(X, \varphi, \Delta) = \sum_{x \in \Period(X, \varphi, n)} \Delta^{(n)}(x)$.
\end{definition}
\begin{example}
If $(X, \varphi, \Delta)$ is a signed Smale space with $\Delta \equiv 1$, then the signed dynamcial zeta function is the dynamical zeta function (see the Introduction):
\[
\zeta_{{\rm dyn}}(z) = \exp \left( \sum_{n=1}^{\infty} \frac{| \Period(X, \varphi, n) |}{n} z^n \right).
\]
For more details on this case, see \cite[Section I.4]{Sma} (and also \cite[Chapter 6]{Put} and references therein).
\end{example}
\begin{example} \label{ClaLefEx}
Suppose $(M,f)$ is an Axiom A diffeomorphism, $(\Omega, f|_{\Omega})$ be the restriction of $f$ to the nonwandering set, and for each $m\in \Period(\Omega, f|_{\Omega}, 1)$,
\[
L(m):={\rm sign}( {\rm det} (I-Df(m): T_m(M) \rightarrow T_m(M))).
\]
The Lefschetz fixed point formula implies that
\[
\sum_{m\in \Period(\Omega, f|_{\Omega}, 1)} L(m) = \sum_{i=0}^{\dim(M)} (-1)^i {\rm Tr}( f_* : H_i(M;\R) \rightarrow H_i(M; \R)).
\]
Moreover, by for example \cite[Proposition 5.7]{FraBook} or \cite[Section I.4]{Sma}, $L(m) = (-1)^q \Delta(m)$ where $q$ is the rank of $E^u$ at the point $m$ and $\Delta$ is as in Example \ref{orientExSign}. From this one obtains the homological zeta function discussed in the Introduction.
\end{example}
\begin{theorem} (compare with \cite[Theorem 6.1.1]{Put}) \label{sigLefThm} \\
For each $k\in \N$,
\[ \sum_{N \in \Z} (-1)^N \Tr \left( ((\varphi^{-1})^s_N \otimes id_{\Q})^n \right)  = \sum_{x \in \Period(X, \varphi, n)} \Delta^{(n)}(x) \]
where $(\varphi^{-1})^s_N \otimes id_{\Q} : H^s_N(X, \varphi, \Delta)\otimes \Q \rightarrow H^s_N(X, \varphi, \Delta)\otimes \Q$ is the map on rationalized homology induced from $\varphi^{-1}$.
\end{theorem}

\begin{definition}
Suppose $(M, f)$ is an Axiom A diffeomorphism satisfying the no-cycle condition, $( \Omega_s)_{s=1}^m$ are the basic sets of $(M,f)$, and $(M_s)_{s=1}^m$ is a filtration associated to the basic sets that satisfies the assumptions in \cite{Bow}. Then we let $f_{even}$ and $f_{odd}$ denotes the map induced by $f$ on 
\[
\bigoplus_{n  \: {\rm even}} H_n(M_s, M_{s-1}) \hbox{ and }\bigoplus_{n  \: {\rm odd}}H_n(M_s, M_{s-1}) 
\]
respectively.

Likewise if $(X, \varphi)$ is a Smale space, we let $\varphi^{-1}_{even}$ and $\varphi^{-1}_{odd}$ denote the map induced by $\varphi^{-1}$ on 
\[
\bigoplus_{n \: {\rm even}} H^s_n(X, \varphi, \Delta)\otimes \Q \hbox{ and }\bigoplus_{n {\rm \: odd}} H^s_n(X, \varphi, \Delta) \otimes \Q
\] 
respectively.
\end{definition}

\begin{theorem} \label{eigValThm}
Suppose $(M,f)$ is an Axiom A diffeomorphism satisfying the no-cycle condition, and $E^u|_{\Omega_s}$ is orientable. Then (using the notation introduced in the paragraph preceding this theorem) there exists a signed Smale space, $(X, \varphi, \Delta)$, such that
\begin{enumerate}
\item[(1)] $(X, \varphi)$ is conjugate to $(\Omega_s, f|_{\Omega_s})$;
\item[(2)] (even case) if $q$ is even, the maps $\varphi^{-1}_{even} \oplus f_{odd}$ and $\varphi^{-1}_{odd}\oplus f_{even}$ have the same nonzero eigenvalues or
\item[(2)] (odd case) if $q$ is odd, the maps $\varphi^{-1}_{even} \oplus f_{even}$ and $\varphi^{-1}_{odd}\oplus f_{odd}$ have the same nonzero eigenvalues. 
\end{enumerate}
\end{theorem}

\begin{corollary} 
Suppose $(M,f)$ is an Axiom A diffeomorphism, $(\Omega, f|_{\Omega})$ is the nonwandering set of $(M,f)$, $E^u|_{\Omega}$ is orientable, and $\Delta: \Omega \rightarrow \{ -1, 1\}$ is defined as in Example \ref{orientExSign}. Let $q: \Omega \rightarrow \{0, 1\}$ be the function defined by $q(x) = {\rm rank}(E^u_x) \hbox{ mod }2$. Then
\begin{enumerate}
\item if $q \equiv 0$, then $\zeta_{{\rm hom}}(z) = \zeta_{(\Omega, \Delta)}(z)$;
\item if $q \equiv 1$, then $\zeta_{{\rm hom}}(z) = 1/\zeta_{(\Omega, \Delta)}(z)$.
\end{enumerate}
\end{corollary}
\begin{proof}
By definition, the signed zeta function is given by 
\[ 
\zeta_{\varphi}(z) = \exp \left( \sum_{n=1}^{\infty} \frac{N_n(X, \varphi, \Delta)}{n} z^n \right) 
\]
where $N_n(X, \varphi, \Delta) = \sum_{x \in \Period(X, \varphi, n)} \Delta^{(n)}(x)$.
If $q\equiv 0$, then 
\[ \sum_{x\in \Period(\Omega, f|_{\Omega}, n)} L(x) = \sum_{x \in \Period(\Omega, f|_{\Omega}, n)} \Delta^{(n)}(x)
\]
while if $q \equiv 1$, then
\[
\sum_{x\in \Period(\Omega, f|_{\Omega}, n)} L(x) = (-1) \sum_{x \in \Period(\Omega, f|_{\Omega}, n)} \Delta^{(n)}(x).
\]
The result then follows.
\end{proof}

\section{Examples}
To conclude the paper, two examples are discussed. These examples point to the possibility of a stronger relationship between the signed version of Putnam's homology and the standard homology of the manifold associated with the Axiom A diffeomorphism. However, such a relationship is (at this point) highly speculative. 
\begin{example}{\bf Shifts of finite type} \\
In \cite{BowFra}, Bowen and Franks prove the following results:
\begin{theorem} (reformulation of \cite[Theorem 3.2]{BowFra})
Suppose $(M,f)$ is an Axiom A diffeomorphism satisfying the no-cycle condition, ${\rm dim}(\Omega_s)=0$, and $q:={\rm rank}(E^u|_{\Omega_s})$. Then there exists signed shift of finite type $(\Sigma_G, \sigma, \Delta_G)$ such that
\begin{enumerate}
\item $(\Sigma_G, \sigma)$ is conjugate to $(\Omega_s, f|_{\Omega_s})$;
\item the maps $\gamma^s_{G, \Delta_G}: \Z G^0 \rightarrow \Z G^0$ and  $f|_{M_s}: H_q(M_s, M_{s-1}) \rightarrow H_q(M_s, M_{s-1})$ are shift equivalent;
\item the maps $\gamma^s_{G, \Delta_G}: D^s_{\Delta_G}(G) \rightarrow D^s_{\Delta_G}(G)$ and $f|_{M_s}: H_q(M_s, M_{s-1}) \rightarrow H_q(M_s, M_{s-1})$ are shift equivalent
\end{enumerate}
where $(M_s)_{s=1}^m$ is a fixed filtration associated to the basic sets, $(\Omega_s)_{s=1}^m$, of $(M,f)$.
\end{theorem}
The reader might notice that Bowen's and Franks' result implies \cite[Theorem 2]{Bow} (stated as Theorem \ref{BowResult} above). However, the proof in \cite{BowFra} uses \cite[Theorem 2]{Bow}.
\end{example}

\begin{example} {\bf Two dimensional hyperbolic toral automorphisms} \label{twoDimTorAut} \\
We give an example in which one can compute the standard homology, Putnam's homology and the relevant actions explicitly. Let 
\[
\varphi=A=\left( \begin{array}{cc} 1 & 1 \\ 1 & 0 \end{array} \right)
\]
and consider the induced action on the two-torus, $\R^2/\Z^2$. In this example, $\Omega$ is the entire manifold and $\Delta$ is the constant function one and $q=1$.

In regards to the standard homology, we have the following
$$H_N(\R^2/\Z^2;\field{R}) \cong \left\{ \begin{array}{ccc}\field{R} & : &  N=0, 2 \\ \field{R}^2 & : & N=1 \\ 0 & :  & \hbox{otherwise} \end{array} \right. $$
and the action is given by the identity on $H_0(\R^2/\Z^2;\field{R})$, $A$ on $H_1(\R^2/\Z^2;\field{R})$, and minus the identity on $H_2(\R^2/\Z^2;\field{R})$.
\par
In regards to Putnam's homology (based on \cite[Example 7.4]{Put}) we have the following 
$$H^s_N(\R^2/\Z^2, A)\otimes \field{R} \cong  \left\{ \begin{array}{ccc}\field{R} & : &  N=-1, 1 \\ \field{R}^2 & : &  N=0 \\ 0 & : & \hbox{otherwise} \end{array} \right. $$
and the action of $((\varphi^{-1})^s)\otimes Id_{\field{R}}$ is given by the identity on $H^s_{-1}(\R^2/\Z^2, A)\otimes \field{R}$, $A$ on $H_0(\R^2/\Z^2,A)\otimes \field{R}$, and minus the identity on $H^s_1((\R^2/\Z^2, A)\otimes \field{R}$ .
\par
Thus, in this very special case, there is an even stronger than predicted by Theorem \ref{eigValThm} relationship between the homology of torus and Putnam's homology of the Smale space $(\R^2/\Z^2, A)$. Namely, they are the same with dimension shift of one (this is exactly the rank of bundle $E^u$ in this case). Moreover, the actions induced by $f$ and $\varphi^{-1}$ are also the same (again with dimension shift).
\end{example}
\begin{acknowledgement}
I thank Magnus Goffeng, Ian Putnam and Robert Yuncken for discussions. In addition, I thank Magnus for encouraging me to publish these results. I also thank the referee for a number of useful suggestions.
\end{acknowledgement}
%

%
%

\end{document}